\documentclass[11pt]{amsart}

\usepackage{times}
\usepackage{geometry}
\usepackage{amssymb}
\usepackage{latexsym,  amsmath, amscd, amsfonts}
\usepackage{graphicx}
\usepackage[percent]{overpic}
\usepackage{pdfsync}
\usepackage{units}
\usepackage{hyperref}
\usepackage{epstopdf}
\usepackage{paralist}

\newtheorem{theorem}{Theorem}

\newtheorem{proposition}[theorem]{Proposition}
\newtheorem{corollary}[theorem]{Corollary}

\theoremstyle{definition}
\newtheorem{definition}[theorem]{Definition}
\newtheorem{example}[theorem]{Example}

\newtheorem{remark}[theorem]{Remark}

\newcommand{\R}{\mathbb{R}}      
\newcommand{\Z}{\mathbb{Z}}
\newcommand{\N}{\mathbb{N}}

\def\Rib{{\operatorname{Rib}}}
\def\Rop{{\operatorname{Rop}}}
\def\Cr{{\operatorname{Cr}}}
\def\Len{{\operatorname{Len}}}
\def\Step{{\operatorname{Step}}}
\def\pr{{\operatorname{pr}}}
\newcommand{\ds}{\displaystyle}

\newcommand{\kw}{K_{w}}

\newcommand{\lw}{L_{w}}

\begin{document}

\title{Ribbonlength and crossing number for folded ribbon knots}

\author{Elizabeth Denne}
\address{Elizabeth Denne, Washington \& Lee University, Department of Mathematics, Lexington VA}
\email{dennee@wlu.edu}
\date{April 2, 20201}                                           

\makeatletter								
\@namedef{subjclassname@2020}{%
  \textup{2020} Mathematics Subject Classification}
\makeatother


\begin{abstract}
We study Kauffman's model of folded ribbon knots: knots made of a thin strip of paper folded flat in the plane. The ribbonlength is the length to width ratio of such a folded ribbon knot. We show for any knot or link type that there exist constants $c_1, c_2>0$ such that the ribbonlength is bounded above by $c_1\Cr(K)^2$, and also by $c_2\Cr(K)^{3/2}$. We use a different method for each bound. The constant $c_1$ is quite small in comparison to $c_2$, and the first bound is lower than the second for knots and links with $\Cr(K)\leq$ 12,748. 
\end{abstract}

\maketitle


\section{Introduction}\label{sectionhistory}

In 2005, L. Kauffman \cite{Kauf05} described a flat knotted ribbon, or a {\em folded ribbon knot}, which can be imagined as a knot tied in a long thin rectangular strip, then flattened to the plane. The ribbon can be thought of as a set of rays parallel to a polygonal knot diagram, where the folds act as mirror segments, and the over and under information is appropriately assigned. This is illustrated for the trefoil knot in Figure~\ref{fig:trefoil-pentagon}. 

\begin{center}
\begin{figure}[htbp]
\begin{overpic}{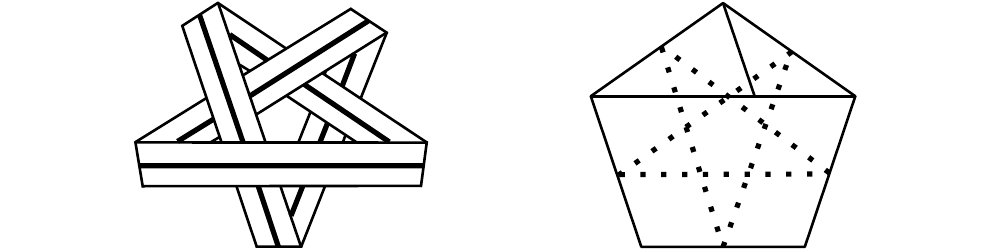}
\end{overpic}
\caption{On the left a folded ribbon trefoil knot. On the right, the folded trefoil knot has been ``pulled tight'', minimizing the folded ribbonlength.}
\label{fig:trefoil-pentagon}
\end{figure}
\end{center}

For a knot (or link) $K$, we denote a folded ribbon knot as $K_w$, where $w$ is the fixed width of the ribbon. We define the {\em folded ribbonlength} $\Rib(\kw)$ to be the length of $K$ divided by the width of the ribbon.  The {\em ribbonlength problem} asks to find the minimum folded ribbonlength needed to tie a folded ribbon knot for a particular knot or link type. There are a number of papers \cite{DeM, DKTZ, Kauf05, KMRT, Tian} which find upper bounds for folded ribbonlength of both specific knots and families of knots.  An overview of this problem can be found in \cite{Den-FRS}.  In addition, we note that the ribbonlength problem is a 2-dimensional version of the {\em ropelength problem} which asks to find the minimum amount of rope needed to tie a knot in a rope of unit diameter.  (See for instance \cite{BS99,CKS,gm,lsdr,Pie}.)  The formal definition of a folded ribbon knot and folded ribbonlength can be found in Section~\ref{sect:definition}, and Section~\ref{sect:local-geom} gives some technical results about the geometry of a fold in a folded ribbon knot.

In this paper we are chiefly interested in the following problem: given a knot or link type $K$ with crossing number $\Cr(K)$, can we find a particular example of $\kw$ with $\Rib(\kw)$ less than some function of $\Cr(K)$?  A detailed discussion of this problem can be found in Section~\ref{sect:crossing-history}.  This question has also been asked for the ropelength problem, and our paper has been motivated by some of the many results about ropelength and crossing number. 

Our first main theorem is found in Section~\ref{sect:arc-presentation}, and uses arc-presentations of knots and links in its proof. Theorem~\ref{thm:crossing-squared} is inspired by the work of Grace Tian \cite{Tian}, and J.~Cantarella, X.W.C. Faber and C. Mullikin \cite{CFM}.
\begin{theorem} \label{thm:crossing-squared}
Any non-split link type $L$  with arc index $\alpha(L)$ contains a folded ribbon link $\lw$ with 
$$\Rib(\lw)\leq \begin{cases} 0.32\Cr(L)^2 + 1.28\Cr(L) + 0.23 & \text{ when $\alpha(L)$ is even,} \\
0.64\Cr(L)^2+2.55Cr(L) + 2.03 & \text{ when $\alpha(L)$ is odd.}
\end{cases}
$$ 
In particular, this bound holds for prime knots.
\end{theorem}
We use the second equation to get similar bounds for composite links in Theorem~\ref{thm:cr-sq-non-prime}.

In Section~\ref{sect:Hamiltonian}, we obtain even better upper bounds on folded ribbonlength by using a particular embedding of a knot in the cubic lattice developed by Y. Diao, C. Ernst and X. Yu \cite{DEY}. We project this embedding to the plane and find the following.
\begin{theorem} \label{thm:ribbonlength}
Any knot or link type $K$ contains a folded ribbon knot $\kw$ with
\begin{enumerate}
\item $ \Rib(K)\leq  9\Cr(K)^{3/2} +8\Cr(K) +6\sqrt{\Cr(K)} +6$, if $K$ is minimally Hamiltonian, and
\item  $\Rib(K)\leq 72\Cr(K)^{3/2} +32\Cr(K)+12\sqrt{\Cr(K)} +6$, otherwise.
\end{enumerate}

\end{theorem}
Here, a knot or link which has a minimal crossing diagram containing a Hamiltonian cycle is called {\em minimally Hamiltonian} (see  Section~\ref{sect:Hamiltonian}).

We note that Theorem~\ref{thm:ribbonlength} gives the best results for folded ribbonlength for knots and links with large crossing number. We have included Theorem~\ref{thm:crossing-squared} in this paper, as $0.64\Cr(L)^2+2.55Cr(L) + 2.03$ is smaller than $72\Cr(K)^{3/2} +32\Cr(K)+12\sqrt{\Cr(K)} +6$ when $\Cr(K)\leq 12,748$. Since just about all knots and links that are usually encountered in physical applications have small crossing number, both theorems are useful.

\section{Folded ribbon knots and ribbonlength}
\subsection{Definition and ribbonlength} \label{sect:definition}

Following standard knot theory texts (for instance \cite{Adams,Crom, Liv})  we assume that a {\em knot} $K$ is an embedding $K:S^1\rightarrow \R^3$ that is ambient isotopic to a polygonal knot.  A {\em link} is a finite disjoint union of knots. A  {\em knot diagram} is a projection of $K$ onto a plane, where gaps are added to show the relative heights of two strands at a crossing.  As mentioned above, our work on folded ribbon knots is inspired by the work of Kauffman \cite{Kauf05}. A detailed discussion of the definition of a folded ribbon knot can be found in \cite{DKTZ}, and we only give an overview of the main ideas here. (Note that an even more general description for smooth ribbon knots in the plane can be found in \cite{DSW}.)

In order to construct a folded ribbon knot, we take a {\em polygonal knot diagram}, and build a ribbon around it.   We denote the vertices of the diagram by $v_1, v_2, \dots, v_n$ and edges $e_1=[v_1,v_2], \dots, v_n=[v_n,v_1]$. If $K$ is an oriented knot, we assume the numbering of vertices and edges follows the orientation. Note that we do not assume that the polygonal knot diagram is regular, instead we remember that the crossing information is assumed to be consistent (see \cite{DKTZ} Definition 2.2). This means, for example, that the unknot can have a polygonal knot diagram with just two edges, where one edge is ``over'' the other edge. This unknot diagram captures the following example: take a ribbon which is an annulus, and then fold it flat with two folds.

We now construct our folded ribbon knot of width $w$ following \cite{DKTZ} Definitions 2.4 and 2.5. We start by defining the {\em fold angle} at vertex $v_i$  to be the angle $\theta_i$ (where $0\leq\theta_i\leq \pi$) between edges $e_{i-1}$ and $e_i$. (This angle is said to be positive when $e_i$ is to the left of $e_{i-1}$ as in Figure~\ref{fig:ribbon-construct} (left), and is negative when $e_i$ is to the right.)    Second, we view the polygonal knot diagram $K$ as the centerline of the folded ribbon knot, with the edges of the ribbon at distance $\nicefrac{w}{2}$ from $K$. Third, the geometry of the situation requires that the fold line is perpendicular to the angle bisector of the fold angle. For example, in Figure~\ref{fig:ribbon-construct} (left), the fold angle is $\theta_i=\angle ECF$, the angle bisector is $DC$, and the fold line is $AB$.  Using the geometry of the figure we see that $\angle GAB=\theta_i/2$ in right triangle $\triangle AGB$. Thus $|AB|=\frac{w}{\cos(\nicefrac{\theta_i}{2})}$ guarantees the ribbon width $|AG|=w$. Fourth, the folded ribbon's boundary is constructed by joining the ends of the fold lines at each vertex. 

In summary, we take an oriented polygonal knot diagram $K$ and constructed a folded ribbon knot as follows.
\begin{compactenum}
\item If the fold angle $\theta_i<\pi$, we place a fold line of length $\frac{w}{\cos(\nicefrac{\theta_i}{2})}$ centered at $v_i$ perpendicular to the angle bisector of $\theta_i$.   If $\theta_i=\pi$, there is no fold line.
\item Add in the ribbon boundaries by joining the ends of the fold lines at $v_i$ and $v_{i+1}$. 
\item The folded ribbon knot inherits an orientation from $K$.
\end{compactenum}

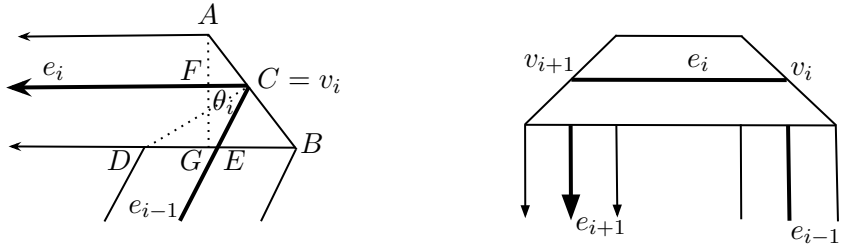
\begin{figure}[htbp]
\begin{center}
\begin{overpic}{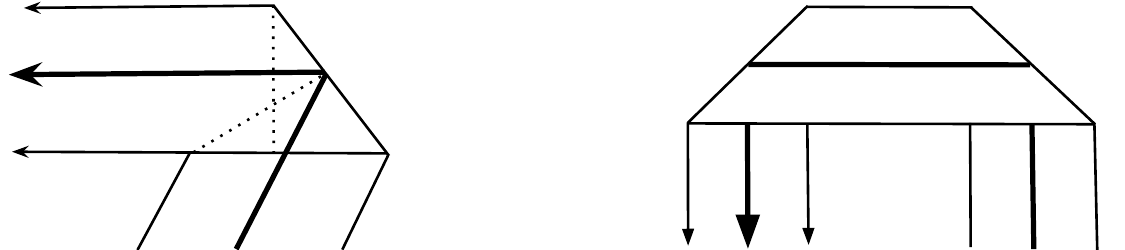}
\put(23,23){$A$}
\put(35,8){$B$}
\put(30,15.5){$C=v_i$}
\put(24.8,13){$\theta_i$}
\put(26,6){$E$}
\put(12.5,6){$D$}
\put(21,16.5){$F$}
\put(21,6){$G$}
\put(5,17){$e_i$}
\put(15,1){$e_{i-1}$}

\put(80,18){$e_i$}
\put(92,-2){$e_{i-1}$}
\put(67,-1){$e_{i+1}$}
\put(92,17){$v_i$}
\put(61,18){$v_{i+1}$}
\end{overpic}
\caption{On the left, a close-up view of a ribbon fold, with fold angle $\theta_i=\angle FCE$. On the right, the construction of the ribbon centered on edge $e_i$. Figure re-used with permission from~\cite{DKTZ}. }
\label{fig:ribbon-construct}
\end{center}
\end{figure}

The astute reader will realize that a folded ribbon knot is really a PL-immersion of an annulus or M\"obius strip into the plane, where the fold lines are the only singularities. 
In addition, we note that very wide ribbons might not be physically possible. Parts of the ribbon cannot pierce the folds, as might happen to the folds in Figure~\ref{fig:ribbon-construct} (right) if the width gets too large.
To prevent this, we assume the fold lines are disjoint.  More formally, we define the notion of an allowed ribbon.

\begin{definition}[\cite{DKTZ} Definition 2.7] Given an oriented polygonal knot diagram $K$, we say the folded ribbon $K_{w}$ of width $w$  is {\em allowed} provided
\begin{enumerate}
\item The ribbon has no singularities (is immersed), except at the fold lines which are assumed to be disjoint.
\item $K_w$ has consistent crossing information, and moreover this agrees
\begin{enumerate} \item with the folding information at each fold, and
\item with the crossing information of the knot diagram $K$.
\end{enumerate}
\end{enumerate}
\end{definition}

From now on {\bf we assume that our folded ribbon knots have an allowed width}. This is a reasonable assumption, since we can always construct folded ribbon knots for ``small enough'' widths (proved in \cite{Den-FRF}).

Given a  folded ribbon knot, we can ask to find the least length of paper needed to tie it. To formalize this idea we define a scale-invariant quantity called {\em folded ribbonlength}.
\begin{definition}[\cite{Kauf05}] The {\em folded ribbonlength} of a folded ribbon knot $K_w$  of width $w$ is defined to be the following: $$ \Rib(K_w) = \frac{\Len(K_w)}{w}.$$
\end{definition}

The {\em (folded) ribbonlength problem} asks us to ``minimize'' the folded ribbonlength of knot or link type. Here, we have two choices. We can fix the width and find the infimum of the length of a knot diagram, or we can fix the length of the knot diagram and find the supremum of the width. As a simple example, consider the folded ribbon unknot (described above) with knot diagram consisting of just two edges. The ribbonlength here is zero. Either the width is fixed and the length $\Len(K)\rightarrow 0$, or the length is fixed and the width $w\rightarrow\infty$.  In practice, for a folded ribbon knot $\kw$, we will usually set the width $w=1$, and compute the length of the the knot diagram. In this case, finding the folded ribbonlength is simple: $\Rib(\kw)=\Len(K)$, and the computation gives an upper bound on the minimum folded ribbonlength for the knot type. As mentioned in the introduction,  there are a number of papers,  including \cite{DeM, DKTZ, Kauf05, KMRT, Tian}, which compute upper bounds for folded ribbonlength of specific knots and of certain infinite families of knots.

\subsection{Local geometry of folded ribbons} \label{sect:local-geom}
Before we get into the main results of the paper, we pause to examine the geometry of a fold in a folded ribbon knot.
In Figure~\ref{fig:AcuteVsObtuse} we see folds  where the fold angle is acute (left) and obtuse (right). In both cases, the {\em fold from angle $\theta$} consists of two overlapping triangles of ribbon, determined by the knot diagram $DC+CH$. Again using the notation in Figure~\ref{fig:AcuteVsObtuse}, we define an {\em extended fold from angle $\theta$} to be the ribbon determined by  $EC+CI$. We will use the folded ribbonlength of an extended fold to get our first upper bound on folded ribbonlength in Section~\ref{sect:arc-presentation}.

\begin{center}
\begin{figure}[htbp]
\begin{overpic}{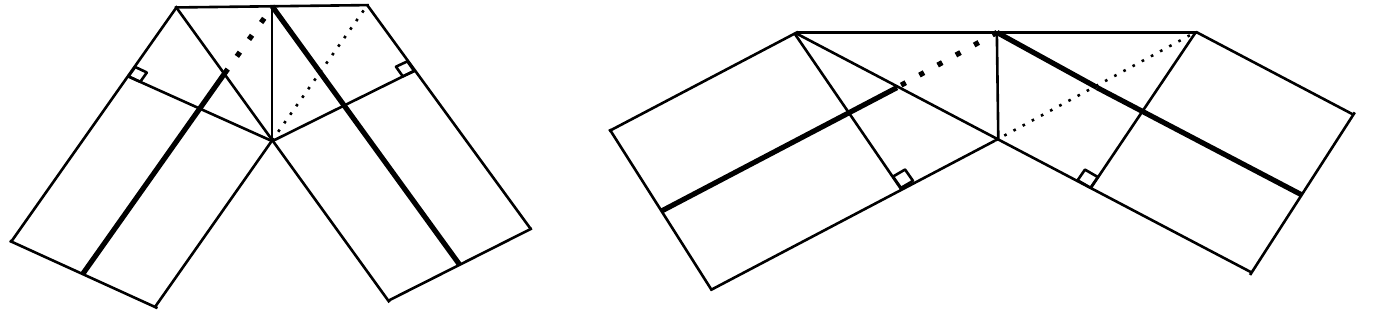}
\put(12,23){$A$}
\put(19,23){$C$}
\put(26,23){$B$}
\put(17,16.5){$D$}
\put(11.5,14){$E$}
\put(18.5,9){$F$}
\put(8,18.5){$G$}
\put(20,16.5){$H$}
\put(23.5, 13){$I$}
\put(57.5,21){$A$}
\put(72,21){$C$}
\put(86,21){$B$}
\put(67.5,15.5){$D$}
\put(57.5,14){$E$}
\put(71.5,9.5){$F$}
\put(64,6.25){$G$}
\put(78,13.5){$H$}
\put(83.4,11){$I$}
\end{overpic}
\caption{On the left a fold where the fold angle $\angle DCH=\theta$ is acute, on the right a fold where the fold angle is obtuse.}
\label{fig:AcuteVsObtuse}
\end{figure}
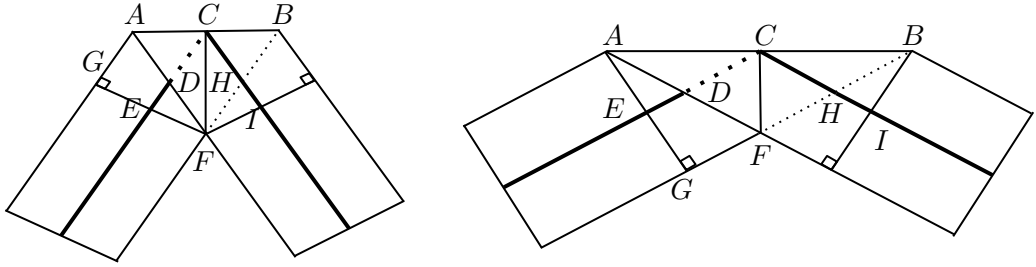
\end{center}

\begin{remark}\label{rmk:fold-length}
In the following, use the notation in Figure~\ref{fig:AcuteVsObtuse}, and assume the fold angle $\angle DCH=\theta$, with $0<\theta<\pi$. Let  $\mathcal{D}=EC+CI$ denote the knot diagram of an extended fold from angle $\theta$. 
We have used the geometry of parallel lines and trigonometric arguments (all found in \cite{Den-FRF} Proposition 14) to show that $|DC|=\ds\frac{w}{2\sin\theta}$, and $|DE|=(\nicefrac{w}{2})\cot\theta$ when $\theta$ is acute, and $|DE|=(\nicefrac{w}{2})\cot(\pi-\theta)$ when $\theta$ is obtuse. This means the length of $\mathcal{D}$ is
$$\Len(\mathcal{D}) = |EC|+|CI|=\begin{cases} w\cot(\nicefrac{\theta}{2})
 & \text{ when $0<\theta\leq \nicefrac{\pi}{2}$,} \\
w\cot(\nicefrac{\pi}{2}-\nicefrac{\theta}{2})& \text{when $\nicefrac{\pi}{2}\leq \theta< \pi$.}
\end{cases}
$$
We then deduced (see \cite{Den-FRF} Corollary 16) that the folded ribbonlength of the extended fold $\mathcal{D}_w$ is the following  
$$\Rib(\mathcal{D}_w)=\begin{cases} \cot(\nicefrac{\theta}{2})
 & \text{ when $0<\theta\leq \nicefrac{\pi}{2}$,} \\
\cot(\nicefrac{\pi}{2}-\nicefrac{\theta}{2})& \text{ when $\nicefrac{\pi}{2}\leq \theta< \pi$.}
\end{cases}
$$
In particular, we note that the folded ribbonlength of an extended fold with fold angle $\theta=\nicefrac{\pi}{2}$ is 1.
\end{remark}

 \subsection{Folded ribbonlength and crossing number}\label{sect:crossing-history}
 
The results in this paper have been inspired by trying to understand the relationship between the folded ribbonlength of a knot $K$ and its crossing number
$\Cr(K)$. The {\em ribbonlength crossing number problem} asks us to find constants $c_1, c_2, \alpha, \beta$ such that 
\begin{equation}\label{eq:bounds}
c_1\cdot( \Cr(K))^\alpha\leq \Rib(K_w)\leq c_2\cdot (\Cr(K))^\beta.
\end{equation}
Y. Diao and R. Kusner \cite{DK} conjecture that $\alpha=\nicefrac{1}{2}$ and $\beta=1$ in Equation~\ref{eq:bounds}, and have work in progress showing $\alpha=\nicefrac{1}{2}$. The rest of this paper is devoted understanding Equation~\ref{eq:bounds}.

Just as the folded ribbonlength problem is a 2-dimensional analogue of the ropelength problem, the conjecture above is the analogue of the ropelength crossing number problem finding upper and lower bounds on ropelength in terms of crossing number. 
There have been many papers making progress this problem, for example \cite{Buck, BS07, CFM, CKS-Nat, CKS, DEKZ, DEPZ, DEY, HKON, HNO}.

We now give evidence for the conjecture that $\alpha=\nicefrac{1}{2}$ by giving examples of infinite families of knots and links where for each $p$ in the range $\nicefrac{1}{2}\leq p\leq 1$, there is a constant $c_p>0$ such that $\Rib(K_w)\leq c_p\cdot \Cr(K)^p$.

\begin{example} Inspired by G. Buck \cite{Buck}, we give example of an infinite family of links with $p=\nicefrac{1}{2}$. In Figure~\ref{fig:Hopf-Chain} (left), we see a folded ribbon link diagram of a Hopf link $A\cup B$. Here, both $A$ and $B$ are unknots with 2 edges that are linked together.  We now change this diagram to give a ``fat'' Hopf link. Suppose that $A$  consists of $n$ nested unknots and that $B$ consists of  $n$ nested unknots, where each of $A$'s unknots is linked with all $n$ of $B$'s unknots and vice versa. To give some more details, set the width $w=1$, and let the length of one component of $A$ to be $\Len(A_i)=2(1+\nicefrac{1}{2^i})\leq 3$. Arrange the $n$-components of $A$ so that $A_2$ is nested inside of $A_1$, $A_3$ is nested inside of $A_2$ and so on. For example, suppose the projection of $A_1$ is $[-\nicefrac{3}{4}, \nicefrac{3}{4}]$ on the $x$-axis, the projection of $A_2$ is $[-\nicefrac{5}{8}, \nicefrac{5}{8}]$ on the $x$-axis, the projection of $A_3$ is $[-\nicefrac{9}{16}, \nicefrac{9}{16}]$ on the $x$-axis, etc.  Similarly, we can arrange the nested components of $B$ to lie along the $y$-axis, and have each component of $B$ be linked with all of the components of $A$. (We will omit a detailed description of the crossing information of the arcs involved.) We compute the folded ribbonlength $\Rib((A\cup B)_w)\leq 2\cdot 3\cdot n=6n$, and the crossing number $\Cr(A\cup B)=2\cdot n\cdot n=2n^2$. Together this means that $\Rib((A\cup B)_w)\leq 3\sqrt{2}\cdot \Cr(K)^{1/2}$.  
\end{example}

\begin{center}
\begin{figure}
\begin{overpic}{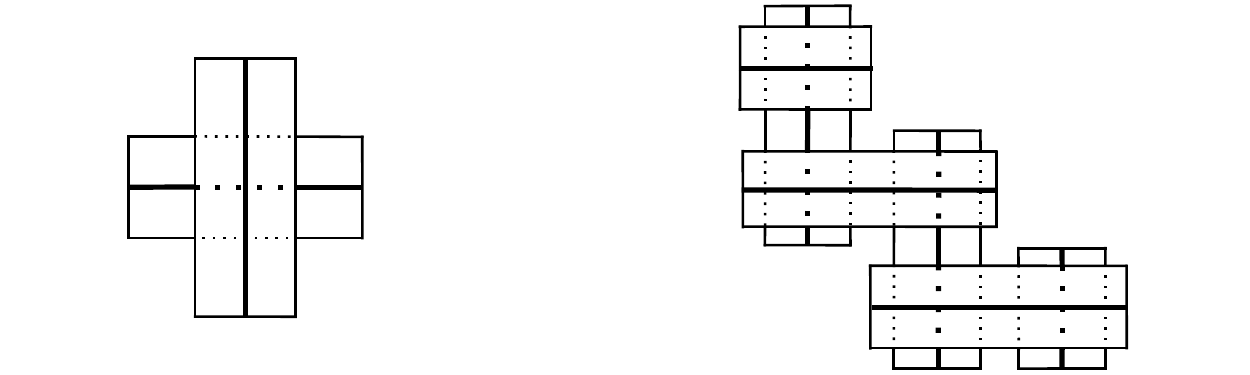}
\put(30,14){$A$}
\put(18,26.5){$B$}
\put(55,24){$A_1$}
\put(63,30.5){$A_2$}
\put(55,14){$A_3$}
\put(73.5,20.5){$A_4$}
\put(65.5,4.5){$A_5$}
\put(83,11){$A_6$}
\end{overpic}
\caption{On the left, a folded ribbon Hopf Link $A\cup B$. On the right, a simple folded ribbon chain link with 6 components.}
\label{fig:Hopf-Chain}
\end{figure}
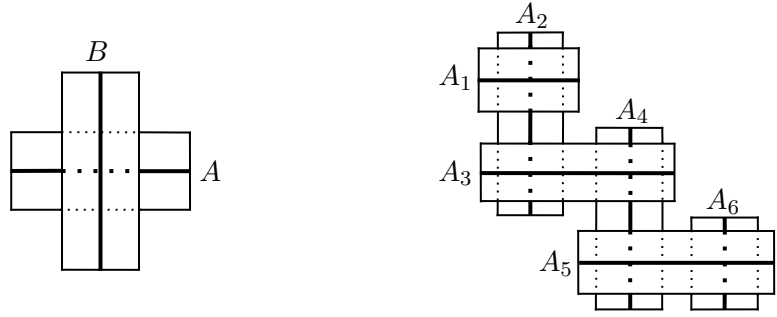
\end{center}

\begin{example} We now follow the ideas of J. Cantarella, R. Kusner, J.M. Sullivan \cite{CKS-Nat} to find an example of an infinite family of links where  $p=1$.  Take a simple chain link $L$ of $n$-components, where each component is a unknot made of 2 edges, and consider the corresponding folded ribbon link $L_w$. For example, a 6-component simple chain folded ribbon link is shown in Figure~\ref{fig:Hopf-Chain} (right). In that figure, we have exaggerated the length of each component to make the link easier to understand. When the width $w=1$, we can arrange for the two end components ($A_1$ and $A_6$ in  Figure~\ref{fig:Hopf-Chain} (right)) to have folded ribbonlength 2, and the other $(n-2)$-components to have folded ribbonlength 4. Thus the folded ribbonlength $\Rib(\lw)=4\cdot (n-2) + 2\cdot 2 = 4n-4$, while the crossing number $\Cr(L)=2(n-1)$. Together this gives $\Rib(\lw) \leq 2\cdot \Cr(L)$. We can take combinations of ``fat'' Hopf links and simple chain links to give many different families of link types $L$. We can choose the combinations such that for any $p$ in the range $\nicefrac{1}{2}\leq p \leq 1$, there is a constant $c_p>0$ and a folded ribbon link $\lw$ with $\Rib(\lw)\leq c_p\cdot \Cr(L)^p$.
\end{example}

In 2020, the author along with J.C. Haden, T. Larsen, and E. Meehan \cite{Den-FRF} proved that for any $p,q\geq 2$, there is a constant $c>0$, such that the $(p,q)$ torus link type $L$ contains a folded ribbon link $\lw$, such that $$\Rib(\lw)=p+q \le c \cdot Cr(L)^\frac{1}{2}.$$ 
Specifically, (in \cite{Den-FRF} Theorem 5) we showed that if $p,q\ge 2$ and $p=aq+b$ for integers $a,b\geq 0$, then  $c=2\sqrt{2}(a+\frac{b}{2})$. This again gives evidence that $\alpha=\nicefrac{1}{2}$ in Equation~\ref{eq:bounds}. 
As a simple example, we apply this result to the infinite family of torus knot types $T(q+1,q)$. For this family we have $a=1$ and $b=1$. Hence for each $q=2,3,4,\dots$, the $(q+1,q)$ torus knot type $K$ contains a folded ribbon knot $\kw$ with $$\Rib(\kw)\le 3\sqrt{2}(Cr(K))^\frac{1}{2}.$$ 
This last example is not unexpected, in that $T(n,n-1)$ torus families are an example  in \cite{CKS-Nat} where ropelength grows as the $\nicefrac{3}{4}$ power of crossing number.

The next question to tackle is to ask what is known about the upper bound $\beta$ in Equation~\ref{eq:bounds}? B. Kennedy, T.W. Mattman, R. Raya, and D. Tating~\cite{KMRT} made a good start on this question by giving upper bounds on folded ribbonlength for the $(p,2)$, $(q+1,q)$, $(2q+1,q)$, $(2q+2,q)$, and $(2q+4,q)$ torus knot families. They combined this, with fact that the crossing number of a $(p,q)$ torus knot is $\min\{p(q-1),q(p-1)\}$,  to show the upper bound on folded ribbonlength is quadratic in crossing number for the $(p,2)$ torus knots, and linear in crossing number for the four other torus knot families.  Again in  2020, we (see \cite{Den-FRF}) gave new constructions of 2-bridge knots, $(2,p)$ torus, twist and pretzel knots. (These were the first folded ribbon constructions for 2-bridge and pretzel knots.) Using these constructions we showed the upper bound of folded ribbonlength is linear in crossing number for these families of knots and links. 

In 2017, Tian~\cite{Tian} provided the first proof that in Equation~\ref{eq:bounds}, $\beta\leq 2$ for all knots and links. She used the fact that every knot and link has a grid diagram associated to it \cite{Cro95, OSZ}. A {\em grid diagram}, with grid number $n$, is an $n\times n$ square grid with $n$ $X$'s and $n$ $O$'s arranged so that every row and column contains exactly one $X$ and one $O$.  The grid index of a knot, $g(K)$, is the minimum grid number needed for the knot, and we know $g(K)\leq \Cr(K)+2$ (see \cite{BP, Bel, Cro98}).  Tian showed that
\begin{equation}
\Rib(K)\leq 2g(K)(g(K)-1)  \leq 2\Cr(K)^2 +6\Cr(K) + 4. \label{eq:Tian}
 \end{equation}

It turns out that we can do better! The fact that every knot has a grid diagram comes from the fact that every knot has an arc-presentation. In Section~\ref{sect:arc-presentation}, we use arc-presentations to get an upper bound on folded ribbonlength which is quadratic in crossing number and has smaller coefficients than those in Equation~\ref{eq:Tian} . Finally, in Section~\ref{sect:Hamiltonian}, we are inspired by a particular embedding of a knot in the cubic lattice by Diao {\em et al.} \cite{DEY} to do even better. We prove that for all knots and links, $\beta\leq \nicefrac{3}{2}$ in Equation~\ref{eq:bounds}.

 
\section{Arc-presentations, ribbonlength and crossing number}\label{sect:arc-presentation}

In 2004, Cantarella, Faber and Mullikin \cite{CFM} use arc-presentations of knots to get an upper bound  on ropelength which is quadratic in crossing number.  Inspired by this work, we use arc-presentations to get an upper bound on folded ribbonlength which is quadratic in crossing number. 

We recall  that an {\em arc-presentation} of a link $L$, is an embedding of $L$ in a finite number of open half-planes around a common axis so that each half-plane meets $L$ in a single simple arc. (See for instance \cite{CFM, Crom}.) The minimum number of pages needed to present a link in this manner is a link invariant called the {\em arc-index} of $L$, and is denoted by $\alpha(L)$. It is known that every link has an arc-presentation \cite{Crom}. In addition we know that for any-non split link has $\alpha(L)\leq \Cr(L)+2$ (with equality when the link is alternating \cite{Cro98}, and with the strict inequality when the link is non-alternating \cite{BP, Bel}). 

Following \cite{CFM}, we can arrange (by isotopy) for $L$ to intersect the axis only at points $1, 2, \dots, \alpha$. We can also arrange for the planes to be located at angle $\theta_i=\frac{2\pi i}{\alpha}$ for $i=0,1,\dots, \alpha-1$, around the axis. 
An arc-presentation of the trefoil knot is shown in Figure~\ref{fig:trefoil-arc-presentation} (left), where the numbers on the arc indicate the half-plane that each arc lies in. 
Figure~\ref{fig:trefoil-arc-presentation} (right) shows the same arc-presentation, but viewed from a point on the axis above the knot. In this {\em spoked form}, the axis is the central point, and each arc is a spoke coming out from it. The numbers at the end of the spoke are the points where the arc intersects along the axis.

\begin{center}
\begin{figure}\begin{overpic}{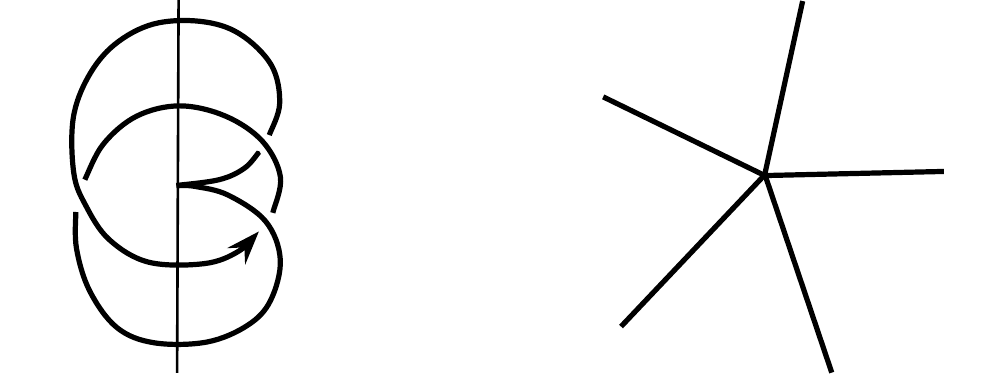}
\put(15.5,36){$a$}
\put(15.5,28){$b$}
\put(15,18){$c$}
\put(15.25,12){$d$}
\put(15.5,4){$e$}
\put(28.5,8){$1$}
\put(29,18){$2$}
\put(29,28){$3$}
\put(6,8){$4$}
\put(5.5,28){$5$}
\put(92,22){$ce$}
\put(76,37){$bd$}
\put(59,29){$ac$}
\put(59,6){$eb$}
\put(84,1){$da$}
\put(82,21.5){$1$}
\put(75,27){$2$}
\put(68,20){$3$}
\put(71,12){$4$}
\put(80,12){$5$}
\end{overpic}
\caption{Two views of the same arc-presentation of a trefoil knot. On the left, an axial form where the numbers on the arc indicate  the half-plane the arc lies in. On the right, a spoked form where each spoke represents an arc and the letters represent the points where the arc intersects the axis.}
\label{fig:trefoil-arc-presentation}
\end{figure}
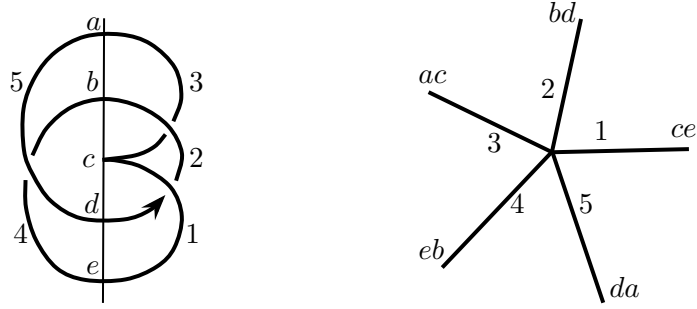
\end{center}

\begin{proposition} \label{prop:spoked}
An arc-presentation of a link $L$ with arc-index $\alpha$ can be realized with folded ribbonlength at most 
$$\Rib(\lw)\leq
\begin{cases} \alpha \cot(\frac{\pi}{\alpha}) & \text{ when $\alpha$ is even}, \\
\alpha \cot(\frac{\pi}{2\alpha}) &  \text{ when $\alpha$ is odd}.
\end{cases}
$$
\end{proposition}

\begin{proof} To compute the ribbonlength we assume that the width $w=1$, and find the length of the link diagram. We use the {\em spoked form} of the arc-presentation for $L$. We can view the folded ribbon traveling along each spoke from the end to the center and out again. The angle between each consecutive spoke is $\nicefrac{2\pi}{\alpha}$.

The ribbonlength of one spoke (from the end to the center) is bounded below by the ribbonlength of {\em half} of an extended fold~$\mathcal{D}$ from angle $\theta$. Recall from Remark~\ref{rmk:fold-length} that
$$\Rib(\mathcal{D})=\begin{cases} \cot(\nicefrac{\theta}{2})
 & \text{ when $0<\theta\leq \nicefrac{\pi}{2}$,} \\
\cot(\nicefrac{\pi}{2}-\nicefrac{\theta}{2})& \text{ when $\nicefrac{\pi}{2}\leq \theta< \pi$.}
\end{cases}
$$
This function increases as $\theta\rightarrow 0$ and as $\theta\rightarrow \pi$. 
The smallest angle we need to worry about is $\theta=\nicefrac{2\pi}{\alpha}$. Here, $\Rib(\mathcal{D})= \cot(\nicefrac{\pi}{\alpha})$ and the ribbonlength of a single spoke is at most $\frac{1}{2}\cdot\cot(\nicefrac{\pi}{\alpha})$. 

We now consider the largest angle that $\theta$ can be. We first note that when the angle between the spokes is $\pi$, there is no fold, and the ribbon just goes straight across.  The largest angle for our fold angle $\theta$ then depends on whether $\alpha$ is even or odd.  When $\alpha$ is even, we find this is
$$\theta = \frac{2\pi}{\alpha}\left(\frac{\alpha}{2} - 1\right) = \pi\left(1-\frac{2}{\alpha}\right),
$$
and $\pi-\theta=\nicefrac{2\pi}{\alpha}$. Here, the ribbonlength of the spoke with this fold angle is  bounded below by $\frac{1}{2}\cdot\cot(\nicefrac{\pi}{\alpha})$, the same as the acute case. This leads us to deduce that when $\alpha$ is even, the ribbonlength of the folded ribbon link is at most
$$\Rib(\lw)\leq 2\alpha\cdot \frac{1}{2}\cdot\cot(\frac{\pi}{\alpha}) =\alpha \cot(\frac{\pi}{\alpha}).
$$

When $\alpha$ is odd, we find the largest fold angle is 
$$\theta = \frac{2\pi}{\alpha}\left(\frac{\alpha-1}{2}\right) = \pi\left(1-\frac{1}{\alpha}\right),
$$
and $\pi-\theta=\nicefrac{\pi}{\alpha}$. Here, the ribbonlength of the spoke with this fold angle is bounded below by $\frac{1}{2}\cdot\cot(\nicefrac{\pi}{2\alpha})$, which is larger than the acute case. For $\alpha$ odd, we deduce that the folded ribbonlength of the folded ribbon link is at most
$$\Rib(\lw)\leq 2\alpha\cdot \frac{1}{2}\cdot\cot(\frac{\pi}{2\alpha}) =\alpha \cot(\frac{\pi}{2\alpha}).
$$
\end{proof}

We are now ready to prove Theorem~\ref{thm:crossing-squared}: any non-split link type $L$  with arc index $\alpha(L)$ contains a folded ribbon link $\lw$ with 
$$\Rib(\lw)\leq \begin{cases} 0.32\Cr(L)^2 + 1.28\Cr(L) + 0.23 & \text{ when $\alpha(L)$ is even,} \\
0.64\Cr(L)^2+2.55Cr(L) + 2.03 & \text{ when $\alpha(L)$ is odd.}
\end{cases}
$$

\begin{proof}[Proof of Theorem~\ref{thm:crossing-squared}.] 
Assume $L$ has arc-index $\alpha(L)$. We know that $\alpha(L)\leq \Cr(L)+2$. Taylor's theorem says $\cot(x)=\frac{1}{x}-\frac{x}{3} - \frac{x^3}{45}- \dots$, and we will only use the first two terms in our approximation. We now combine these facts along with the formulas from  Proposition~\ref{prop:spoked}. For $\alpha(L)$ even we get
\begin{align*} \Rib(L)\leq \alpha\cot(\frac{\pi}{\alpha}) & \leq \frac{\alpha^2}{\pi} - \frac{\pi}{3}
\\ &\leq \frac{1}{\pi}(\Cr(L)+2)^2-  \frac{\pi}{3}
\\ & = \frac{1}{\pi}\Cr(L)^2 + \frac{4}{\pi}\Cr(L) + \frac{4}{\pi}-\frac{\pi}{3}
\\ & \leq 0.32\Cr(L)^2 + 1.28\Cr(L)+0.23.
\end{align*}
A similar computation for $\alpha(L)$ odd gives
\begin{align*} \Rib(L)\leq \alpha\cot(\frac{\pi}{2\alpha}) & \leq \frac{2}{\pi}\Cr(L)^2 + \frac{8}{\pi}\Cr(L) + \frac{8}{\pi}-\frac{\pi}{6}
\\ & \leq 0.64\Cr(L)^2 + 2.55\Cr(L)+2.03.
\end{align*}
In each case the decimal approximation was chosen to be greater than the constants in the equations.
\end{proof}

When we apply Theorem~\ref{thm:crossing-squared} to a knot $K$, we see this a huge improvement over Tian's \cite{Tian} bound of $\Rib(K)\leq 2(\Cr(K))^2 +6\Cr(K) + 4$.  

\begin{theorem} \label{thm:cr-sq-non-prime}
Any non-split link type $L$ with prime components $L_1, L_2, \dots, L_n$, contains a folded ribbon link $\lw$ with 
$$Rib(\lw)\leq 0.64\sum_{i=1}^n \Cr(L_i)^2 + 2.55\sum_{i=1}^n \Cr(L_i) + 2.03n.$$
\end{theorem}

\begin{proof} Assume that we have found minimal arc-presentations of all of the component links. Then arrange the component links in their spoked form and use the bounds on ribbonlength as in Theorem~\ref{thm:crossing-squared}.  We will show that for any prime links $L_1$ and $L_2$ we can arrange $L_1\#L_2$ so that 
$$\Rib(L_1\#L_2)\leq \Rib(L_1)+\Rib(L_2).$$
 To complete the proof, we apply the larger of the two bounds from Theorem~\ref{thm:crossing-squared} to each $L_i$ and then sum them up.

We start by arranging that the spoked form of $L_1$ is placed below the spoked form of $L_2$. The topmost arc of $L_1$, respectively the bottom arc of $L_2$, is a path from one end point of a spoke to the center then out to an end point of a different spoke. We replace each of these arcs by a straight line segment of length $\ell_1$ and $\ell_2$ respectively. By the triangle inequality, these straight line segments have length less than or equal to the spokes they replace. See Figure~\ref{fig:Link-Sum}.

Now rotate and translate $L_2$ so the straight line segments are collinear and share one endpoint. To complete that construction of $L_1\#L_2$, we remove the straight line segments. We then join $L_1$ and $L_2$ at the shared end point and by a new straight line segment of length $|\ell_1-\ell_2|$ between the remaining end points of the removed straight line segments. At each step,  observe that we reduce the ribbonlength. 
\end{proof}

\begin{center}
\begin{figure}[htbp]
\begin{overpic}{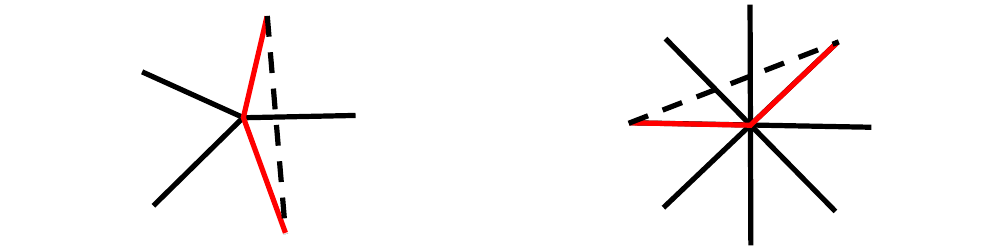}
\put(10,15){$L_1$}
\put(85,15){$L_2$}
\end{overpic}
\caption{The top (respectively bottom) pairs of spokes of $L_1$ (resp. $L_2$)  are shown in red. In the first step of constructing $L_1\# L_2$, these spokes are replaced by the dotted straight-line segments.}
\label{fig:Link-Sum}
\end{figure}
\end{center}

Recall that Theorems~\ref{thm:crossing-squared} and \ref{thm:cr-sq-non-prime} were heavily inspired by the results of Cantarella {\em et. al} \cite{CFM} relating ropelength to crossing number. We note that other authors \cite{HKON,HNO} improved the coefficients in the ropelength case even further by embedding the knot in the cubic lattice. We have chosen not to apply their results to this case, since we have discovered an even better bound for ribbonlength as shown in the next section.


\section{Hamiltonian knot projections, ribbonlength and crossing number}
\label{sect:Hamiltonian}

In this section we use a construction by Diao, Ernst and Yu~\cite{DEY}, in order to show that the folded ribbonlength of any knot or link $K$ is bounded above by a multiple of $\Cr(K)^{3/2}$.  The key result of~\cite{DEY}  is the proof that any knot or link can be embedded the cubic lattice in a particular way.  To get our bounds on folded ribbonlength in terms of crossing number, we simply take this cubic lattice embedding of a knot and project to the $xz$-plane. While this idea is very simple, we will spend time explaining the notation and key steps of the construction in \cite{DEY} to allow the reader to easily understand where our formulas come from.

\subsection{Notation} We start by reminding the reader that the {\em cubic lattice} is the three dimensional grid $(\Z\times\R\times\R)\cup(\R\times\Z\times\R)\cup (\R\times\R\times \Z)$. The {\em planar lattice} is the two dimensional grid  $(\Z\times\R)\cup(\R\times\Z)$. We will eventually measure the length of paths in the cubic lattice and in the planar lattice. We will use $ax$ (where $a\in \N$) to mean we have taken a total of $a$ steps in the $x$ direction. Unlike~\cite{DEY}, we won't keep track of steps in the positive or negative direction, just the total number of steps taken. So $2x$ could mean 2 steps in the same $x$-direction, or one step in the positive and one step in the negative $x$-direction. Similarly, we will use $ay$ (and $az$) for the total of $a$ steps in the $y$-direction (and $z$-direction).  In the work that follows, we will need to keep track of these three numbers for a path, as well as the total length of a path.
\begin{definition} \label{def:step}
Suppose $\mathcal{P}$ is a path in the cubic lattice and the total number of steps taken in $x$, $y$ and $z$ directions is $ax+by+cz$.  We define the {\em step number} of $\mathcal{P}$ to be $\Step(\mathcal{P})=ax+by+cz$. The total length of $\mathcal{P}$ in the cubic lattice is then $\Len(\mathcal{P})=a+b+c$.
\end{definition}

We now review  notation and background from \cite{DEY}.
Given any knot or link $K$, we can find a reduced\footnote{A {\em reduced} knot diagram contains no reducible crossings. A crossing is {\em reducible} if there exists a circle in the projection plane meeting the crossing transversely, but not meeting the knot diagram in any other point.}, regular projection of $K$ denoted $\pr(K)$.  This projection can be viewed as a 4-regular plane graph\footnote{A {\em 4-regular plane graph} is a planar graph where all vertices have degree 4.}  $G$, where we have treated the crossings of $\pr(K)$ as vertices, and the arcs of $\pr(K)$ joining these crossings as edges. In this case, we say that $G$ is an {\em RP-graph of $K$}. If the projection $\pr(K)$ has $\Cr(K)$ number of crossings, then we say $G$ is a {\em minimal RP-graph} of $K$.  We let $V(G)$ denote the set of vertices of $G$, and $E(G)$ denote the set of edges of $G$. Recall that a {\em Hamiltonian cycle} in a graph $G$ is a cycle that contains all vertices of $G$. We say a knot $K$ is {\em Hamiltonian} if there  is a knot $K'$ equivalent to $K$ such that $K'$ has a Hamiltonian RP-graph.  Diao {\em et al.}
then prove that

\begin{theorem}[\cite{DEY} Theorem 3.6] \label{thm:Hamiltonian}
Every knot or link $K$ admits a Hamiltonian RP-graph with at most $4\cdot\Cr(K)$ vertices.
\end{theorem}

Before we go any further, we outline our method (inspired by \cite{DEY}) of  finding bounds on folded ribbonlength.
\begin{compactdesc}
\item[Stage 1] Given any knot or link $K$, we find its Hamiltonian RP-graph $G$.
\item[Stage 2] We embed $G$ into the cubic lattice, and denote this embedding by $F$. (Diao {\em et al.} \cite {DEY} prove that their construction guarantees that $F$ is ambient isotopic to $G$.)
\item[Stage 3] We modify the embedded graph $F$ to give an embedded knot in the cubic lattice which is equivalent to $K$.
\item[Stage 4] We project the embedded knot in the cubic lattice to the $xz$-plane to give a polygonal knot diagram.
\item[Stage 5] We use the polygonal diagram as the core-curve of a folded ribbon knot of width $w=1$.
\end{compactdesc}

We have discussed Stage 1 above. We will describe Stage 2 in detail in Section~\ref{sect:embed-graph}, and Stages 3, 4 and 5 will be described in Section~\ref{sect:project}.

Before we get into the details of Stage 2, we need to introduce some more notation (following \cite{DEY} Notation 4.2). Let $G$ be a Hamiltonian RP-graph of a knot or link $K$ in the plane $z=0$, and let $C$ be a Hamiltonian cycle in $G$.  Let $n=|V(G)|$ be the number of vertices of $G$. We let $v_1, v_2, \dots, v_n$ represent a cyclic order of the vertices of $G$ as we walk along the Hamiltonian cycle $C$. We also let $k=\lceil \sqrt{n} \rceil$. 

Observe that the Hamiltonian cycle $C$ is a simple closed curve in the plane, so divides the plane into a bounded and unbounded region. The edges in the set $E(G)\setminus E(C)$ are divided into two groups. Those in the bounded regions called {\em B-edges}, and those in the unbounded region called {\em U-edges}.  In Section~\ref{sect:embed-graph} we will first embed $C$ in the cubic lattice, and then embed the $U$- and $B$-edges into the cubic lattice.

Exactly how the $U$- and $B$-edges are arranged depends on the other edges in $E(G)\setminus E(C)$.  Recall that $G$ is a 4-regular graph from a reduced, regular knot projection. Hence, at each vertex $v_\ell$, there are 2 distinct edges of $G$ incident with $v_\ell$ which are not on $C$.  Following the notation in \cite{DEY} Section 4.5, we say $v_\ell$ is of {\em type} (a) if one edge is a $U$-edge and one edge is a $B$-edge. We say $v_\ell$ is of {\em type} (b) if these two edges are both $U$-edges. We say $v_\ell$ is of {\em type} (c)  if these two edges are both $B$-edges.

\subsection{Embed Hamiltonian RP-graphs in the cubic lattice}\label{sect:embed-graph}

In this section we describe Stage 2: how to embed $G$, the Hamiltonian RP-graph of $K$, into the cubic lattice. We do this in in 3 steps.  The first step is to embed the vertices of $G$ in the cubic lattice in the plane $z=0$.  Informally, we embed the vertices of $G$ into a square region by snaking the vertices to and fro as shown in Figure~\ref{fig:Hamiltonian-cycle}. More formally, following \cite{DEY} Section 4.4, and we embed $v_1, \dots v_k$ in this order along the $y$-axis starting with $v_1$ at $(0,3,0)$ and increasing the $y$-coordinate 3 units at a time, ending with $v_k$ at $(0,3k,0)$. We then embed $v_{k+1}, \dots, v_{2k}$ on the line determined by $x=3, z=0$ by starting with $v_{k+1}$ at $(3,3k,0)$ and decreasing the $y$-coordinate 3 units at a time, ending with $v_{2k}$ at $(3,3,0)$. This construction is repeated until all vertices are arranged. It is possible that the $(k-1)$st column does not contain any vertices of $G$, for example when $n=6$, or $n=11$.

\begin{figure}[htbp]
\begin{overpic}[scale=1.5]{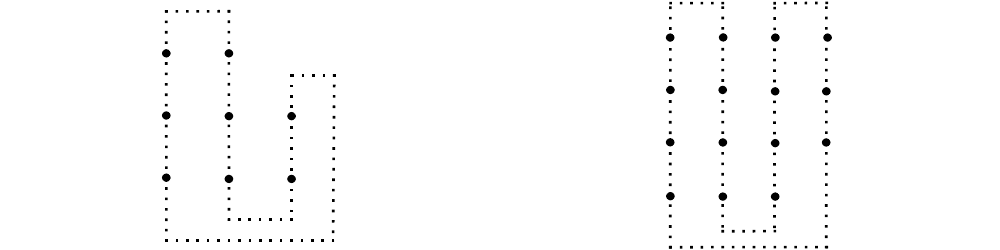}
\put(13.5,6){$v_1$}
\put(13.5,13){$v_2$}
\put(13.5,19){$v_3$}
\put(20,19){$v_4$}
\put(20,13){$v_5$}
\put(20,6){$v_6$}
\put(26,6){$v_7$}
\put(26,13){$v_8$}
\put(64,4.5){$v_1$}
\put(64,10){$v_2$}
\put(64,15){$v_3$}
\put(64,20.5){$v_4$}
\put(69.25,20.5){$v_5$}
\put(69.25,15){$v_6$}
\put(69.25,10){$v_7$}
\put(69.25,4.5){$v_8$}
\put(74.5,4.5){$v_9$}
\put(73.75,10){$v_{10}$}
\put(73.75,15){$v_{11}$}
\put(73.75,20){$v_{12}$}
\put(79,20){$v_{13}$}
\put(79,15){$v_{14}$}
\put(79,10){$v_{15}$}
\end{overpic}
\caption{The embedding of the vertices of the Hamiltonian cycle $C$ in the $z=0$ plane for $n=8$ and $n=15$.}
\label{fig:Hamiltonian-cycle}
\end{figure}
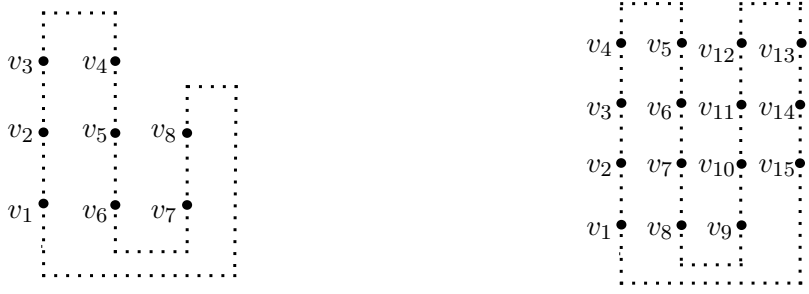

The second step of the embedding of $G$ into the cubic lattice is to embed the edges of $C$ in the cubic lattice between the planes $z=-1$ and $z=1$. The following discussion summarizes \cite{DEY} Section 4.5. We first connect the vertices $v_\ell$ to $v_{\ell+1}$ for $\ell = 1, 2,\dots, n-1$. The possible paths taken are shown in \cite{DEY} Figure~13. In each case, we will take the longest possible path between the vertices. The longest path occurs when vertex $v_\ell$ is of type (b) or (c), and so has an extra 2 steps in the $z$ direction. We first take those paths between consecutive vertices connected by steps only in the $y$-direction (for example from $v_1$ to $v_2$, or from $v_2$ to $v_3$). Such paths have step number at most $3y+2z$.  If we fix $k$ and consider possible values for $n$, then for $(k-1)^2<n\leq k^2-k$, there are $(n-1)-(k-2)=n-k+1$ such paths, and for $k^2-k+1\leq n\leq k^2$ there are $(n-1) - (k-1)=n-k$ such paths.  Next, we note there are at most $(k-1)$-paths connecting vertices which are connected with steps in the $x$ direction. The step number here is at most $3x+4y+2z$. (Notice that these paths travel in the $x$-direction either along the line given by $y=1$, $z=0$, or the line $y=3k+2$, $z=0$.)

The final path to consider is the path connecting $v_n$ to $v_1$.   Observe that by definition of $k$, we have $n=i\cdot k + j$, for some $k-2\leq i \leq k-1$, and $1\leq j\leq k$. The shape of the path from  $v_n$ to $v_1$ depends on whether $i$ is even or odd. It turns out that the longest lattice path occurs when $i$ is even and $v_n$ is of type (b) or (c), so we restrict our attention to that case.  An example of the outline of this path is given in Figure~\ref{fig:Hamiltonian-cycle} (left). This path connects $v_n$ at $(3i,3j,0)$ to $v_1$ at $(0,3,0)$ (and travels the longest distance in the $x$-direction along the $x$-axis). In order to clarify \cite{DEY} Section 4.5, we see the path passes through these vertices in order: $(3i,3j,0)$, $(3i,3j,\pm 1)$, $(3i, 3j+2,\pm 1)$, $(3i,3j+2,0)$, $(3i+2, 3j+2,0)$, $(3i+2,0,0)$, $(0,0,0)$ then $(0,3,0)$.   Using the bounds from the definition of $i$ and $j$, we find this gives us a step number of at most 
\begin{align*}(3i+4)x + (3j+7)y + 2z & \leq [(3(k-1)+4]x + (3k+7) y +2z
\\ & = (3k+1)x +(3k+7)y + 2z.
\end{align*}

We now add together all the different sub-paths in the embedded Hamiltonian cycle $C$ to find that the step number of $C$ is at most
\begin{align*} \Step(C) & \leq (n-k+1)(3y+2z) + (k-1)(3x+4y+2z) + (3k+1)x +(3k+7)y + 2z
\\ & = (3k-3 +3k+1)x + (3n -3k + 3+  4k-4 + 3k+7)y + (2n-2k + 2 + 2k-2 +2)z
\\ & = (6k-2)x + (3n +4k +6)y + (2n+2)z.
\end{align*}

In the Remark in Section 4.5 in \cite{DEY}, Diao {\em et al.} bound the lattice length of $C$ from above: $\Len(C)\leq 5n+ 11k$. They then use this bound in their later equations for lattice length and ropelength. We will keep our bound on the step number of $C$ unchanged,  so the coefficients in our equations later on will be a little different from those in \cite{DEY}.

The third step of the embedding of $G$ into the cubic lattice is to embed the $U$-edges of $G$ into the half-space $z\geq 0$ and the $B$-edges of $G$ into the half-space $z\leq 0$. This requires many careful arguments, and the embedding of $U$-edges is completely described in \cite{DEY} Sections 4.6 through 4.10. The construction for $B$ edges is completely analogous and so is omitted.

As before, we are looking for the embedding of a U-edge which gives the longest possible lattice length. Let $e$ be a $U$-edge which starts in column $i$ and ends in column $j$. Following \cite{DEY}, we define $J(e)=|i-j|$ to be the {\em jump number of $e$}. Diao {\em et. al} discuss 3 cases: the first where $J(e)=0$ and the end points of $e$ are in the same column (\cite{DEY} Section 4.7), the second where $J(e)=1$ (\cite{DEY} Section 4.8), and the third where $2\leq J(e)\leq k-1$ (\cite{DEY} Section 4.10). It turns out that the longest possible $U$-edge comes from the third case. Many things need to be considered, including the number of and way that other edges of $U$ are nested between the edge $e$ and the Hamiltonian cycle $C$.   The path is completely described in \cite{DEY} Section 4.10 with moves numbered (0) through (12). In this construction, we see that there are at most 
\begin{compactitem}
\item $3(k-1)=3k-3$ $x$-steps from moves (2), (4), (6), (8), (10);
\item $2(4k)+2 = 8k+2$ $y$-steps from moves (0), (3), (9), (12);
\item $2(3k-1)=6k-2$ $z$-steps from moves (1), (5), (7), (11).
\end{compactitem}
Thus any individual $U$- or $B$-edge $e$ has step number at most 
$$ \Step(e) \leq (3k-3)x +(8k+2)y +(6k-2)z.$$

We recall from the definition of the Hamiltonian cycle $C$, the total number of $U$- and $B$-edges is $n$.
Following \cite{DEY}, we let $F$ denote the embedding of the Hamiltonian RP-graph $G$ in the cubic lattice. To summarize, $F$ is the union of the embedding of the Hamiltonian cycle $C$,  and the embedding of all of the $U$- and $B$-edges in the cubic lattice.  In \cite{DEY} Theorem 4.13, Diao {\em et al.} prove that the lattice graph $F$ is ambient isotopic to the RP-graph $G$ of $K$.

We can now give a more detailed version of \cite{DEY} Theorem 4.14. (Note that our coefficients end up being a little different since we do not bound $\Step(C)$ from above as in \cite{DEY}.)

\begin{theorem} \label{thm:graph-length}
Let $G$ be a Hamiltonian RP-graph with $n$ vertices, then $G$ can be embedded onto a lattice graph $F$ such that  
$$\Step(F) \leq (3n\lceil \sqrt{n}\rceil -3n +6\lceil \sqrt{n}\rceil -2) x + (8n\lceil \sqrt{n}\rceil+5n+4\lceil \sqrt{n}\rceil +6)y + ( 6n\lceil \sqrt{n}\rceil + 2)z,$$
and 
$$\Len(F) \leq 17n\lceil \sqrt{n}\rceil +2n +10\lceil \sqrt{n}\rceil +6.$$
\end{theorem}

\begin{proof} In order to get an upper bound on the total number of steps for $F$, we combine the step bounds on the Hamiltonian cycle $C$, and all of the $U$- and $B$-edges. We also recall that $k=\lceil \sqrt{n}\rceil$.
\begin{align*} \Step(F) & \leq (6k-2)x + (3n +4k +6)y + (2n+2)z + n[(3k-3)x +(8k+2)y +(6k-2)z]
\\ & = (6k-2 +3nk -3n)x + (3n +4k +6 + 8nk + 2n)y + (2n +2 + 6nk-2n)z 
\\ & = (3nk -3n+6k -2)x + (8nk+ 5n + 4k +6)y + (6nk+2)z
\\ & = (3n\lceil \sqrt{n}\rceil -3n +6\lceil \sqrt{n}\rceil -2) x + (8n\lceil \sqrt{n}\rceil+5n+4\lceil \sqrt{n}\rceil +6)y + ( 6n\lceil \sqrt{n}\rceil + 2)z
\end{align*}
Set $x=y=z=1$ in $\Step(F)$ to get the length bound.
\end{proof}

\subsection{Immerse the knot or link in the planar lattice}\label{sect:project}

In the previous two sections we saw how to take a knot or link $K$, find a Hamiltonian RP-graph $G$, and embed this graph into the cubic lattice. We denoted this embedding by $F$.  We now take $F$ and make changes to it to get bounds on folded ribbonlength of knots. 

We now proceed to Stage 3 of our method: embedding a knot or link $K$ into the cubic lattice. We do this by taking the embedded graph $F$ just constructed and resolving the vertices so there are two distinct strands. To do this we follow the argument given in \cite{DEY} Theorem 5.1. At each vertex $v_\ell$, there are two edges incident to $v_\ell$ along the $y$-direction. We move these two edges one unit in the $x$-direction (either to the left of right) as shown in \cite{DEY} Figure 18.  The edges incident to $v_\ell$ in the vertical $z$-direction are unchanged. This gives an embedded curve $K'$ in the cubic lattice.  At each vertex we can choose to move the edges to the left or right in such a way so that $K'$ is equivalent to the original knot $K$.  We now recover a more detailed version of \cite{DEY} Theorem 5.1 (where again our coefficients end up being a little different due to the bounds on $\Step(C)$). 

\begin{theorem} \label{thm:cubic-length}
Let $K$ be a knot or a link, and assume $K$ has a Hamiltonian RP-graph $G$ with $n$ vertices. Then we can embed $K$ into the cubic lattice such that 
$$\Step(K) \leq  (3n\lceil \sqrt{n}\rceil -n +6\lceil \sqrt{n}\rceil -2) x + (8n\lceil \sqrt{n}\rceil+5n+4\lceil \sqrt{n}\rceil +6)y + ( 6n\lceil \sqrt{n}\rceil + 2)z,$$
and 
$$\Len(K) \leq 17n\lceil \sqrt{n}\rceil + 4n +10\lceil \sqrt{n}\rceil +6.$$
\end{theorem}

\begin{proof} Resolving each vertex as described above adds an extra $2x$-steps. In other words, the embedded cubic lattice path $K'$ equivalent to $K$ has an extra  $(2n)x$-steps in comparison with $F$ in Theorem~\ref{thm:graph-length}.
\end{proof}

We are finally ready for Stages 4 and 5. At this point our discussion departs from the work in~\cite{DEY}. We project the embedding of $K$ in the cubic lattice to the $xy$-, or $yz$-, or $xz$-plane in order to get a polygonal knot diagram in the planar lattice. The height of the edges in the direction of projection gives the crossing information for the corresponding knot diagram. Any knot diagram in the planar lattice gives a corresponding folded ribbon knot with width $w=1$. This means that the length of the polygonal knot diagram in the planar lattice also gives the folded  ribbonlength. When we look at the embedding of $K$ in the cubic lattice in Theorem~\ref{thm:cubic-length},  it turns out that projecting onto the $xz$-plane gives the shortest length in the planar lattice (and hence shortest ribbonlength). We thus deduce the following results.

\begin{theorem} \label{thm:lattice-length}
Let $K$ be a knot or a link, and assume $K$ has a Hamiltonian RP-graph $G$ with $n$ vertices. Then $K$ has a polygonal knot diagram in the planar lattice with length at most
$$  \Len(K) \leq 9n^{3/2} +8n+6\sqrt{n} +6.$$
\end{theorem} 
\begin{proof} From Theorem~\ref{thm:cubic-length} we have an embedding of $K$ in the cubic lattice. Now project to the plane $y=0$. As described above, this gives a polygonal knot diagram in the planar lattice. Next set $x=z=1$ and $y=0$ in the expression for $\Step(K)$ in Theorem~\ref{thm:cubic-length}. Then 
\begin{align*} \Len(K)& \leq 9n\lceil \sqrt{n}\rceil -n +6\lceil \sqrt{n}\rceil 
\\ & < 9n(\sqrt{n}+1) -n +6(\sqrt{n}+1)
\\ & = 9n^{3/2} +8n+6\sqrt{n} +6.
\end{align*}
\end{proof}

We are now very close to proving Theorem~\ref{thm:ribbonlength}.

\begin{corollary} \label{cor:lattice-length}
Let $K$ be a knot or a link, and assume $K$ has a Hamiltonian RP-graph $G$ with $n$ vertices. Then $K$ has a polygonal knot diagram in the planar lattice with folded ribbonlength at most
$$  \Rib(K) \leq 9n^{3/2} +8n+6\sqrt{n} +6.$$
\end{corollary}

\begin{proof}[Proof of Theorem~\ref{thm:ribbonlength}] Recall from Theorem~\ref{thm:Hamiltonian} that the number of vertices $n$ of a Hamiltonian RP-graph of a knot or link $K$ satisfies $n\leq 4\cdot \Cr(K)$.  Apply this to Corollary~\ref{cor:lattice-length} to prove that the folded ribbonlength of $K$ satisfies
\begin{enumerate}
\item $ \Rib(K)\leq  9\Cr(K)^{3/2} +8\Cr(K) +6\sqrt{\Cr(K)} +6$, if $K$ is minimally Hamiltonian, and
\item  $\Rib(K)\leq 72\Cr(K)^{3/2} +32\Cr(K)+12\sqrt{\Cr(K)} +6$, otherwise.
\end{enumerate}
\end{proof}

\section{Conclusion}
All of the work in this paper is aimed at better understanding the ribbonlength crossing number problem. That is, finding constants in Equation 1 from Section~\ref{sect:crossing-history};
$$c_1\cdot( \Cr(K))^\alpha\leq \Rib(K_w)\leq c_2\cdot (\Cr(K))^\beta.$$

As discussed in Section~\ref{sect:crossing-history}, we know that the infinite families of knots and links (2-bridge, twist, torus, certain pretzel knots) examined so far have representatives where $\beta=1$.

If we play with the bounds from Equation (2) in Theorem~\ref{thm:ribbonlength} in WolframAlpha, we see that $72\Cr(K)^{3/2} +32\Cr(K)+12\sqrt{\Cr(K)} +6 < 80\Cr(K)^{3/2}$ when $\Cr(K)\geq 20$. 
This gives constants $c_2=80$ and $\beta=\nicefrac{3}{2}$ (provided $\Cr(K)\geq 20$).

Previously in Theorem~\ref{thm:crossing-squared} we showed that any non-split link $K$ satisfies
\begin{equation} \Rib(K)\leq 0.64\Cr(K)^2+2.55Cr(K) + 2.03.\label{eq:squared}
\end{equation}

How does this folded ribbonlength compare with our results from Theorem~\ref{thm:ribbonlength}? A computation in WolframAlpha reveals that  the values of Equation~\ref{eq:squared} are smaller than (2) in Theorem~\ref{thm:ribbonlength} only when $\Cr(K)\leq 12,748$. Thus both Theorems~\ref{thm:crossing-squared} and ~\ref{thm:ribbonlength} are useful when considering folded ribbonlength.

We remind the reader that our results were inspired by work done on the ropelength problem. Many other authors have attempted to look at this problem, and currently the best known bounds relating ropelength to crossing number come from a paper by Y. Diao, C. Ernst, A. Por and U. Ziegler~\cite{DEPZ}. There, they show for any knot $K$ that there exists a constant $c>0$ such that $\Rop(K)\leq c\cdot\Cr(K)\ln^5(\Cr(K))$. That is, ropelength is almost linear.  We have not attempted to see if these results can be applied to the folded ribbonlength problem. 

Finally, here are just a few open questions for the folded ribbonlength problem (many are the direct analogies for the ropelength problem).
\begin{enumerate}
\item Is there a family of knots and links for which folded ribbonlength has super-linear growth in crossing number? More formally, for any $1<p\leq \nicefrac{3}{2}$, is there a constant $c>0$ and an infinite family of knots and links such that for any representative  $K$ of the family, $\Rib(\kw)\geq c\cdot\Cr(K)^p$? 
\item Is it possible to do better than Theorem~\ref{thm:ribbonlength}? Equivalently, is it possible to improve the embedding of a knot in the planar lattice such that $\Rib(\kw)\leq O(\Cr(K)^p)$ for some constant $1\leq p<\nicefrac{3}{2}$?
\item  Is it possible to prove that folded ribbonlength is at most linear in crossing number for all knots and links? That is, for all knots and links, is there a constant $c>0$ such that $\Rib(\kw)\leq c\cdot\Cr(K)$?
\end{enumerate}


\section{Acknowledgments}
My colleague Jason Cantarella has been a source of many helpful comments and often acted as a sounding board for my ideas. Thank you for your continuing support. I also thank Michael Bush for triple checking all the routine algebra in the lattice length equations.

All of my work on folded ribbon knots has been done while advising undergraduate research projects on the topic. I wish to give a special thank you to all of my undergraduate research students who have contributed to the project over the years.
\begin{compactitem}
\item Shivani Aryal and Shorena Kalandarishvili: funded by Smith College's 2009 Summer Undergraduate Research Fellowship program.\item Eleanor Conley, Emily Meehan and Rebecca Terry: funded by the Center for Women in Mathematics at Smith College Fall 2011, which was funded by NSF grant DMS 0611020. 
\item Mary Kamp and Catherine (Xichen) Zhu: funded by 2015 Washington \& Lee Summer Research Scholars program.
\item Corinne Joireman and Allison Young: funded by 2018 W\&L Summer Research Scholars program.
\item John Carr Haden and Troy Larsen: funded by 2020 W\&L Summer Research Scholars program.
\end{compactitem}


\bibliographystyle{amsalpha}

\end{document}